\documentclass{amsart}
\usepackage{amsmath,amssymb,bm,amscd}
\usepackage[dvips]{graphicx}
\usepackage{color}
\usepackage{amsthm}%
\numberwithin{equation}{section}
\theoremstyle{definition}
\newtheorem{defn}{Definition}[section]

\newtheorem{rem}[defn]{Remark}

\newtheorem{conjecture}[defn]{Conjecture}
\theoremstyle{plain}
\newtheorem{thm}[defn]{Theorem}
\newtheorem{prop}[defn]{Proposition}
\newtheorem{lem}[defn]{Lemma}
\newtheorem{cor}[defn]{Corollary}

\newtheorem{question}[defn]{Question}

\newcommand{\R}{\mathbb{R}}

\title[Characterization of positive links and the $s$-invariant for links]{
Characterization of positive links and the $s$-invariant for links}
\author{Tetsuya Abe} 
\author{Keiji Tagami}
\date{\today}
\address{Osaka City University Advanced Mathematical Institute, 
3-3-138 Sugimoto, Sumiyoshi-ku
Osaka 558-8585 Japan}
\email{tabe@sci.osaka-cu.ac.jp}
\address{
Department of Mathematics, Faculty of Science and Technology, Tokyo University
of Science, Noda, Chiba, 278-8510, Japan
}
\email{tagami\_keiji@ma.noda.tus.ac.jp}
\begin{document}
\maketitle
\begin{abstract}
We characterize positive links in terms of strong quasipositivity, homogeneity and the value of Rasmussen and Beliakova-Wehrli's $s$-invariant. 
We also study almost positive links,
in particular, determine the $s$-invariants of 
almost positive links.
This result  suggests that all almost positive links might 
be strongly quasipositive.
On the other hand, it implies that 
almost positive links are never homogeneous links.
\end{abstract}
%
\section{Introduction}
\par
A link is called {\it positive} if it has a diagram with only positive crossings,
which is defined combinatorially.
On the other hand,
Nakamura \cite{Nakamura} and Rudolph   \cite{Rudolph}
proved that 
positive links are strongly quasipositive links,
which are defined geometrically.
It is natural to consider the following question. 
\begin{question}
Find  differences between positive links and strongly quasipositive links.
\end{question}
Cromwell \cite{homogeneous} introduced a class of links, which is called homogeneous links. 
A homogeneous link is a generalization of 
positive links from the combinatorial view points.
Baader \cite{baader} proved that 
a knot is positive if and only if 
it is strongly quasipositive and homogeneous, 
answering Question 1.1 in the case of knots
(see also \cite{abe2}). 
One can obviously apply Baader's proof to the case of links and obtain the following.
\begin{thm}[\cite{baader}]\label{baader}
A non-split link is positive if and only if it is
strongly quasipositive and homogeneous.
\end{thm} 
We generalize the above theorem as follows:
\begin{thm}
\label{theorem:positivity2}
Let $L$ be a non-split link with $\sharp L$ components. Then $(1)$--$(4)$ are equivalent. \\
$(1)$ $L$ is positive,\\
$(2)$ $L$ is homogeneous and strongly quasipositive,\\
$(3)$ $L$ is homogeneous, quasipositive and $g_{*}(L)=g(L)$,\\
$(4)$ $L$ is homogeneous and $s(L)=2g_{*}(L)+\sharp L-1=2g(L)+\sharp L-1$, \\
where $s(L)$ is Rasmussen and Beliakova-Wehrli's $s$-invariant of $L$, $g_{*}(L)$ is the four-ball genus of $L$ and $g(L)$ is the three-genus of $L$. 
\end{thm}
For the definition of Rasmussen and Beliakova-Wehrli's $s$-invariant, see \cite{cat_colored, rasmussen1}. 
Theorem~$\ref{theorem:positivity2}$ is a generalization of \cite[Theorem~1.3]{abe2}. 
We prove Theorem~$\ref{theorem:positivity2}$ in Section~$\ref{sec:positive}$. 
The key of the proof is the computation of the $s$-invariants
of homogeneous links (see Sections~\ref{sec:homogeneous}-\ref{sec:positive}). 
\par
In this paper, 
we also study almost positive links. 
An {\it almost positive link} is a non-positive link which is represented by a diagram with exactly one negative crossing. 
In general, it is hard to distinguish almost positive links from positive links. 
We consider the following question. 
\begin{question}
Find similarities and differences between positive links and almost positive links. 
\end{question}
There are some similarities between them
(see \cite{cochran-gompf}, \cite{homogeneous}, \cite{positive_signature}, \cite{almostpositive}, \cite{almost_stoimenow}, and \cite{negative_signature}).
One of the interesting and expected similarities is Stoimenow's question: 
\begin{question}[{\cite[Question 4]{stoimenow1}}]\label{conj:stoimenow}
Is any almost positive link strongly quasipositive, or at least quasipositive? 
\end{question}
We give an evidence towards an affirmative answer to Question~\ref{conj:stoimenow} as follows.
\begin{thm}
\label{main2}
Let $L$ be a non-split  link with 
$\sharp L$ components.
If $L$ is almost positive or strongly quasipositive, then
\[s(L) =2g_{*}(L)+\sharp L-1=2g(L)+\sharp L-1. 
\]
\end{thm} 
Moreover, we determine the $s$-invariant of an almost positive link in terms of its almost positive diagram (see Theorem~$\ref{link-rem}$). 
We also confirm Question~\ref{conj:stoimenow} for 
 fibered almost positive knots  (Theorem~$\ref{thm:fiber-almost}$)
and almost positive knots up to $12$ crossings in Section~$\ref{appendix}$. 
\par
On the other hand, 
there are some differences between 
positive links and almost positive links. 
In this paper, we give a significant difference between them. 
In fact, we prove the following. 
\begin{cor}%
\label{cor2}
Any almost positive link is not homogeneous. 
\end{cor}
Note that positive links are homogeneous
and this corollary follows from Theorems~\ref{theorem:positivity2} and \ref{main2}, 
see Section~$\ref{sec:almost-positive-link}$. 
Moreover, using Corollary \ref{cor2}, 
we give infinitely many knots which are pseudo-alternating and are not homogeneous
(which are counterexamples of Kauffman's conjecture (Conjecture~\ref{kauffman_conjecture})). 
\begin{prop}%
\label{cor3}
There are infinitely many knots which are pseudo-alternating and are not homogeneous. 
\end{prop}

\par
This manuscript is organized as follows: 
In Section~$\ref{sec:homogeneous}$, we recall Kawamura-Lobb's inequality and homogeneous links. 
In Section~$\ref{sec:strong-quasi}$, we recall strongly quasipositive links. 
In Section~$\ref{sec:positive}$, we give a characterization of positive links. 
In Section~$\ref{sec:almost-positive-link}$, we compute the $s$-invariants of almost positive links. 
As a corollary, we prove that any almost positive link is not homogeneous (Corollary~$\ref{cor2}$). 
In Section~$\ref{appendix}$, 
we consider the strong quasipositivities of almost positive knots with up to $12$ crossings. 
In Section~$\ref{sec:app}$, 
we give infinitely many counterexamples of Kauffman's conjecture on pseudo-alternating links and alternative links. 
\par
Throughout this paper, we call Rasmussen and Beliakova-Wehrli's invariant by $s$-invariant. Also, we assume that all links and diagrams are oriented. 
\par
\section{Kawamura-Lobb's inequality and homogeneous links}\label{sec:homogeneous}
In this section, we recall homogeneous links and their properties. 
\subsection{Kawamura-Lobb's inequality for the $s$-invariant}
In this subsection, we recall Kawamura-Lobb's inequality
for the $s$-invariant. 
\par
Here we recall some definitions.
For a connected diagram $D$,
let  $w(D)$ be the writhe of $D$,  
$O(D)$ the number of Seifert circles for $D$  
and $O_{+}(D)$ (resp.~$O_{-}(D)$)  the number of connected components of the diagram 
obtained from $D$ by smoothing all negative (resp.~positive) crossings of $D$.
Kawamura \cite{kawamura2} and Lobb \cite{lobb1} gave estimations for the $s$-invariant of a link 
independently, which turned out to be the same estimation. 
The statement is the following.
\begin{thm} [{\cite{kawamura2}, \cite[Theorem~1.10]{lobb1}}] \label{them:kawaumura}
Let $D$ be a connected diagram of a link $L$.
Then, we obtain 
\[ w(D)-O(D)+1+2(O_{+}(D)-1) \le s(L) \le w(D) +O(D)-1-2(O_{-}(D)-1).\]
\end{thm}

\subsection{Homogeneous links}
For a fixed diagram $D$,
we consider when the upper bound and the lower bound of Kawamura-Lobb's inequality coincide.
The answer is when $D$ is homogeneous. 
In particular, the $s$-invariant of any homogeneous link is determined by its homogeneous diagram and Kawamura-Lobb's inequality. 
This result was given by the first author \cite{abe2}. 
In this section, we see this result in terms of $*$-product. 
\par
We recall the definition of $\ast$-product of diagram (see also \cite{homogeneous}). 
The Seifert circles of a diagram is divided into two types:
a Seifert circle is of \textit{type~1}
 if it does not contain any other Seifert circles in one of the 
complementary regions of the Seifert circle in $\R ^{2}$,
otherwise it is of \textit{type~2}.
Let $D \subset \R^{2}$ be a knot diagram and $C$ 
 a type~2 Seifert circle of $D$.
Then $C$ separates $\R^{2}$ into two components $U$ and $V$ such that
$U \cup V =\R^{2}$ and $U \cap V = \partial U=\partial V=C$.
Let $D_1$ and $D_2$ be the diagrams obtained form $D \cap U$ and 
$D \cap V$ by adding suitable arcs from $C$, respectively.
Then $C$ \textit{decomposes} $D$ into
a {\it $*$-product} of $D_1$ and $D_2$, which is denoted by $D=D_1 * D_2$.
We call this decomposition a {\it $*$-product decomposition} of $D$.
A diagram is \textit{special} if $D$ has no Seifert circles
of type~2.
It is not hard to see that a special positive (or negative) diagram is alternating and a special alternating diagram is positive or negative. 
Clearly, any diagram is decomposed into 
\[ D_1 * D_2* \cdots * D_n, \]
where $D_i$ is a special diagram.
\par
For a diagram, any simple closed curve in $\mathbf{R}^{2}$ meeting the diagram transversely at two points cuts the diagram into two parts. 
A diagram is {\it strongly prime} if one of such parts has no crossing for any simple closed curve meeting the diagram transversely at two points (see \cite{knot-gtm}). 
If $D$ is not strongly prime,
 $D$ is represented as a connected sum of non-trivial diagrams $D_{1}$ and $D_{2}$ on $\mathbf{R}^{2}$. 
Then we also write $D=D_1 * D_2$.
Any diagram $D$ is decomposed into 
\[ D_1 * D_2* \cdots * D_n, \]
where $D_i$ is a strongly prime diagram.
\par
As a result, any diagram is decomposed into 
\[ D_1 * D_2* \cdots * D_n, \]
where $D_i$ is a special and strongly prime diagram. 
This $*$-product decomposition of $D$ depends only on $D$. 
On the other hand, for given diagrams $D_{1}$ and $D_{2}$, a $*$-product $D_{1}*D_{2}$ is not well defined. 
Throughout this section, if we write $D=D_{1}*D_{2}$, it is one of the diagrams which have such a $*$-product decomposition. 
\par
Let $L(D)$ and $U(D)$ be the lower bound and the upper bound of Kawamura-Lobb's inequality, respectively. 
Namely, 
\[ L(D)=w(D)-O(D)+1+2(O_{+}(D)-1), \]
\[ U(D)= w(D) +O(D)-1-2(O_{-}(D)-1).\]
\begin{lem} \label{lem:star}
Let $D_1*D_2$ be a connected link diagram which has a $\ast$-product decomposition of two diagrams $D_{1}$ and $D_{2}$. 
Then, we have 
\begin{align*}
L(D_1*D_2)&=L(D_1)+L(D_2), \\ 
U(D_1*D_2)&=U(D_1)+U(D_2). 
\end{align*} 
\end{lem}
\begin{proof}
It follows from the following facts: 
\begin{align*}
\omega (D_1*D_2)&= \omega(D_1) +\omega(D_2), \\
O(D_1*D_2)&= O(D_1) +O(D_2)-1, \\
O_{+}(D_1*D_2)&= O_{+}(D_1) +O_{+}(D_2)-1, \\
O_{-}(D_1*D_2)&= O_{-}(D_1) +O_{-}(D_2)-1. 
\end{align*}
\end{proof}
A diagram is {\it homogeneous} if it has a $*$-product decomposition whose factors are some special alternating diagrams.
A {\it homogeneous link} is a link represented by a homogeneous diagram (\cite{homogeneous}, and see also \cite{baader}, \cite{banks1} and \cite{manch1}). 
Note that positive or negative links are homogeneous. 
\par
Let $\Delta(D)$ be the half of the difference between $U(D)$ and $L(D)$, that is, 
\begin{align*}
\Delta (D):=(U(D)-L(D))/2=O(D)+1-O_{+}(D)-O_{-}(D). 
\end{align*}
The following result ensures that $\Delta(D)=0$ for any homogeneous diagram $D$. 
\begin{thm}\label{theorem:Rasmussen-homog}
Let $D=D_1*D_2* \cdots * D_n$ be a connected homogeneous diagram of a link $L$, where each $D_{i}$ is a special alternating diagram. 
Then we obtain $\Delta(D)=0$. 
\end{thm}
\begin{proof}
We have $\Delta(D_{i})=0$ for $i=1, \dots, n$ since any special alternating diagram is positive or negative. 
By Lemma~$\ref{lem:star}$ we obtain 
\[ L(D) = \sum L(D_{i}) = \sum U(D_{i})=U(D).\]
\end{proof}
\begin{cor}
Let $D=D_1*D_2* \cdots *D_n$ be a connected homogeneous diagram of a link $L$, where each $D_{i}$ is a special alternating diagram. 
Then, we have 
\[ s(L)=\sum _{i=1}^{n} {s(D_i)} =L(D)=U(D).\] 
In particular,  $s(\overline{L})=-s(L)$.
\end{cor}
The following theorem was proved by the first author. 
From Theorem~$\ref{theorem:Rasmussen-homog}$ and Theorem~$\ref{lem:Rasmussen-homog}$ below, we see that $\Delta(D)=0$ if and only if $D$ is homogeneous. 
\begin{thm}[\cite{abe2}]\label{lem:Rasmussen-homog}
Let $D$ be a connected diagram of a link $L$.
If $\Delta(D)=0$, then  $D$ is homogeneous. %
\end{thm}
%
%
\subsection{Kawamura's inequality}\label{sec:good-diagram}
Kawamura \cite{kawamura1} gave another estimation for the $s$-invariant for any non-positive and non-negative knot. 
The first author \cite{abe1} gave an alternative proof of the estimation by using state cycles of the Lee homology. 
In this section, we determine the difference between Kawamura-Lobb's inequality and Kawamura's inequality.  
\par
Let $D$ be a diagram of a link. 
A Seifert circle of $D$ is {\it strongly negative (resp.~positive)} if it is not adjacent to any positive (resp.~negative) crossing. 
Let $O_{<}(D)$ (resp.~$O_{>}(D)$) be the number of the strongly negative (resp.~positive) circles of $D$. 
Then we obtain the following Kawamura's inequality. 
\begin{thm}[\cite{kawamura1}, see also \cite{abe1}]
Let $D$ be a connected diagram of a non-positive and non-negative link $L$. Then we obtain 
\begin{align*}
w(D)-O(D)+1+2O_{<}(D)\leq s(L)\leq w(D)+O(D)-1-2O_{>}(D). 
\end{align*}
\end{thm}
\begin{rem}
Kawamura \cite{kawamura1} and the first author \cite{abe1} only proved the above theorem for the $s$-invariants of  knots. However, 
both of their methods can be applied to the $s$-invariants for links. 
\end{rem}
Any strongly negative (resp.~positive) circle of $D$ is a connected component of the diagram obtained from $D$ by smoothing all negative (resp.~positive) crossings of $D$. 
Hence, if D is neither positive nor negative, we obtain 
\begin{align*}
O_{<}(D)+1&\leq O_{+}(D), \\
O_{>}(D)+1&\leq O_{-}(D), 
\end{align*}
in particular, we notice that Kawamura-Lobb's inequality is sharper than Kawamura's inequality. 
\par 
Let $D$ be a connected link diagram and $S_{D}$ be the Seifert graph of $D$, that is, the vertices of $S_{D}$ correspond to the Seifert circles of $D$ and two vertices are connected by an edge with the label $+$ (resp.~$-$) if there is a positive (resp.~negative) crossing of $D$ which is adjacent to the circles corresponding to the two vertices. 
Let $S_{D}^{+}$ (resp.~$S_{D}^{-}$) be the graph obtained from $S_{D}$ by removing all the edges with the label $-$ (resp.~$+$) and all the vertices corresponding to the strongly negative (resp.~positive) circles of $D$. 
If $D$ is positive (resp.~negative), the graph $S_{D}^{-}$ (resp.~$S_{D}^{+}$) is empty. 
Then we have the following. 
\begin{lem}
Let $D$ be a connected link diagram. Then we obtain 
\begin{align*}
O_{<}(D)+|S_{D}^{+}|&=O_{+}(D), \\
O_{>}(D)+|S_{D}^{-}|&= O_{-}(D), 
\end{align*}
where $|S_{D}^{+}|$ and $|S_{D}^{-}|$ is the number of the components of $S_{D}^{+}$ and $S_{D}^{-}$, respectively. 
\end{lem}
\begin{proof}
From the definition, $O_{+}(D)$ is the number of the components of the graph obtained from $S_{D}$ by removing all the edges with the label $-$. It is equal to the number of the strongly negative circles of $D$ and the components of $S_{D}^{+}$. Hence we obtain the first equality. By the same discussion, we have the second one. 
\end{proof}
\begin{cor}\label{cor1}
For any diagram $D$, the graph  $S_{D}^{+}$ (resp.~$S_{D}^{-}$) is connected and not empty if and only if $O_{<}(D)+1=O_{+}(D)$ (resp.~$O_{>}(D)+1=O_{-}(D)$). 
\end{cor}
\begin{rem}\label{rem1}
From Theorems~$\ref{theorem:Rasmussen-homog}$ and $\ref{lem:Rasmussen-homog}$, for a link diagram $D$, the lower bound and the upper bound of Kawamura-Lobb's inequality are equal if and only if $D$ is homogeneous. 
On the other hand, from Corollary~$\ref{cor1}$,  the lower bound and the upper bound of Kawamura's inequality are equal if and only if $D$ is homogeneous, and $S_{D}^{+}$ and $S_{D}^{-}$ are connected and non-empty. 
Such a diagram has a $*$-product decomposition whose factors are one positive diagram and one negative diagram. 
In \cite[Remark I.26]{lewark}, Lewark called such a diagram {\it good diagram}. 
\end{rem}
\section{The $s$-invariants of strongly quasipositive links}\label{sec:strong-quasi}
In this section, we give a computation of the $s$-invariant of strongly quasipositive links. 
Recall that, for $n \in \mathbb{Z}_{>0}$, the \textit{$n$-braid group} $B_n$, is a group which has the following 
presentation. 
\begin{equation*}
\left< \sigma_1,  \sigma_2, \dots ,  \sigma_{n-1}  \left| 
\begin{array}{cc}
\sigma_t \sigma_s = \sigma_s \sigma_t & (|t-s| >1)\\ 
\sigma_t \sigma_s \sigma_t = \sigma_s \sigma_t \sigma_s & (|t-s| =1)
\end{array}
\right\rangle \right. .
\end{equation*}
Rudolph introduced the concept of a strongly quasipositive  link
(see \cite{Rudolph4}) as follows: 
For $0<i\leq j-1<n$, we define {\it positive embedded band} $\sigma_{i,j}$ as 
\[ \sigma_{i,j} :=(\sigma_{i}, \cdots, \sigma_{j-2})(\sigma_{j-1})
(\sigma_{i}, \cdots, \sigma_{j-2})^{-1}, 
\]
and 
\[ \sigma_{j-1,j} :=\sigma_{j-1}. 
\] 
A link is \textit{strongly quasipositive} 
if it is represented by the closure of a braid of the form
\[ \beta =\prod_{k=1}^{m} \sigma_{i_{k},j_{k}}.\]
Let $L$ be a strongly quasipositive link represented by the closure of $\beta$. 
Then $L$  bounds a surface $F$ in $S^3$
naturally, called a \textit{quasipositive surface} (see Figure~$\ref{positive_surface}$). 
The Euler characteristic $\chi(F)$ of the surface is equal to 
$n-m$, where $n$ is the number of strands of $\beta$, and $m$ is the number of the positive embedded bands in $\beta$. 
\begin{figure}[htbp]
\begin{center}
\includegraphics[scale=0.5]{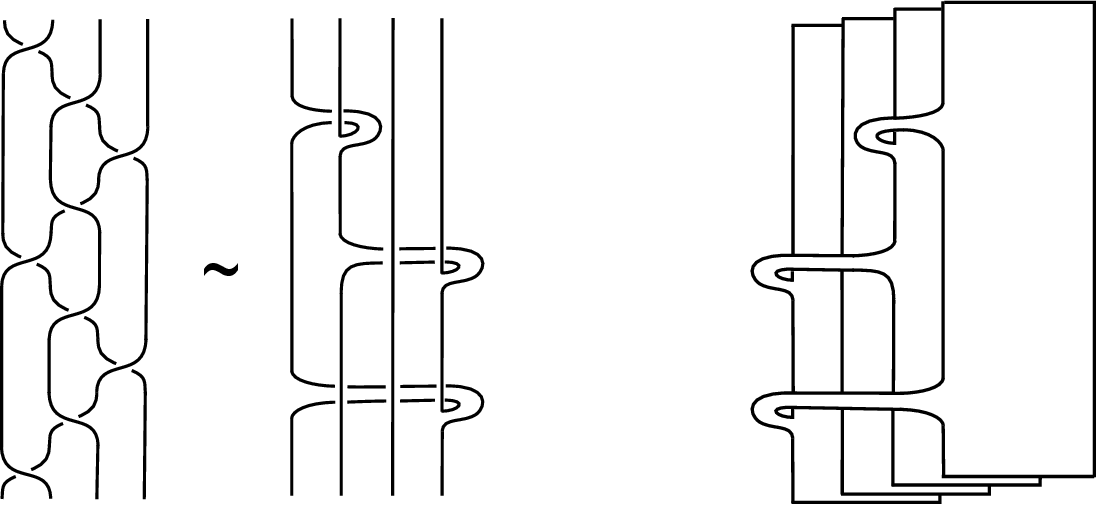}
\end{center}
\caption{An example of a quasipositive surface. The closure of $\sigma_{1}\sigma_{2,4}\sigma_{1,4}$ bounds the right quasipositive surface.} 
\label{positive_surface}
\end{figure}
\par
For a strongly quasipositive knot $K$, Livingston \cite{Liv1} and Shumakovitch \cite{Shumakovitch} proved that 
\begin{align*}
\tau(K)=s(K)/2=g_{*}(K)=g(K)=g(F), 
\end{align*}
where $\tau(K)$ is the Ozsv{\' a}th-Szab{\' o}'s $\tau$-invariant of $K$ (see \cite{tau-inv1} and \cite{rasmussen-thesis}) and $F$ is a quasipositive surface for $K$.
These results are easily generalized to the $s$-invariant for links.
\begin{thm}[\cite{Liv1}]  \label{theorem:positivity}
Let $L$ be a non-split strongly quasipositive link with $\sharp L$ components.
Then 
\begin{align*}
s(L) =2g_{*}(L)+\sharp L-1=2g(L)+\sharp L-1=1-\chi(F),
\end{align*}
where $F$ is a quasipositive surface bounded by $L$. 
\end{thm}
\begin{rem}
In general, Theorem~$\ref{theorem:positivity}$ does not hold for split links. 
In fact, if $L$ is $2$-component unlink, $s(L)=-1$ and $2g_{*}(L)+\sharp L-1=1$. 
\end{rem}
\begin{rem}\label{rem:quasi1}
A link is \textit{quasipositive} 
if it is the closure of a braid of the form
\[ \beta =\prod_{k=1}^{m} \omega_{k} \sigma_{i_{k}} \omega_{k}^{-1},\]
where $\omega_{k}$ is a word in $B_{n}$.
Let $K$ be a quasipositive knot. 
Then $\tau(K)=s(K)/2=g_{*}(K)$.
This is due to Plamenevskaya \cite{Plamenevskaya2}
 and Hedden \cite{Hedden2}
for $\tau$, and
Plamenevskaya \cite{Plamenevskaya} and Shumakovitch \cite{Shumakovitch} for $s$. 
By the same discussion, we obtain the following: 
Let $L$ be a quasipositive link with $\sharp L$ components. Then we obtain $s(L) = 2g_{*}(L)+\sharp L-1$. 
\end{rem}
%
%
\section{Characterization of positive links}\label{sec:positive}
In this section, we prove characterizations of positive links. 
\par
\begin{lem}  \label{lem:positivity}
Let $D$ be a connected reduced homogeneous diagram of a link $L$ with $\sharp L$ components.
If $s(L)=2g(L)+\sharp L-1$, then 
$D$ has no negative crossings. 
\end{lem}
\begin{proof}
Let $D$ be a connected reduced homogeneous diagram of $L$.
Then the genus of $L$ is realized by 
the genus of the surface constructed 
by applying Seifert's algorithm to $D$ (see \cite{homogeneous}).
Therefore, we obtain 
\[ 2g(L)=2-\sharp L+c(D)-O(D),\]
where $c(D)$ denotes the number of crossings of $D$. 
By Theorem \ref{theorem:Rasmussen-homog}, we have
\[s(L)= w(D)-O(D)+2O_{+}(D)-1.\]
By the assumption, $s(L)=2g(L)+\sharp L-1$.
This implies that $O_{+}(D)-1=c_{-}(D)$,
where $c_{-}(D)$ denotes the number of negative crossings of $D$.
If there exists a non-nugatory negative crossing of $D$,
then $O_{+}(D) -1< c_{-}(D)$.
This contradicts the fact that $O_{+}(D)-1=c_{-}(D)$.
Therefore $D$ has no negative crossing.
\end{proof}
\begin{thm}[Theorem~\ref{theorem:positivity2}]
Let $L$ be a non-split link with $\sharp L$ components. Then $(1)$--$(4)$ are equivalent. \\
$(1)$ $L$ is positive.\\
$(2)$ $L$ is homogeneous and strongly quasipositive.\\
$(3)$ $L$ is homogeneous, quasipositive and $g_{*}(L)=g(L)$.\\
$(4)$ $L$ is homogeneous and $s(L)=2g_{*}(L)+\sharp L-1=2g(L)+\sharp L-1$.
\end{thm}
\begin{proof}
$(1) \Rightarrow  (2)$ A positive link is strongly quasipositive
(see \cite{Nakamura} and \cite{Rudolph}) and homogeneous.\\
$(2) \Rightarrow  (3)$
If $L$ is strongly quasipositive, obviously $L$ is quasipositive. 
Moreover, from Theorem~$\ref{theorem:positivity}$, we have $g_{*}(L)=g(L)$. 
\\
$(3) \Rightarrow  (4)$ 
Since $L$ is a quasipositive link,
$s(L)=2g_{*}(L)+\sharp L-1$ (see Remark~$\ref{rem:quasi1}$).
By the assumption, $g_{*}(L)=g(L).$
Therefore $s(L)=2g_{*}(L)+\sharp L-1=2g(L)+\sharp L-1$.\\
$(4) \Rightarrow  (1)$
By Lemma \ref{lem:positivity}, 
a homogeneous diagram of $L$ with $s(L)=2g(L)+\#L-1$ is a positive diagram. 
\end{proof}
\begin{cor} \label{cor:}
Let $L$ be an alternating  link $L$ with $\sharp L$ components.
Then $L$ is positive if and only if
$s(L)=2g(L)+\sharp L-1$.
\end{cor}
\begin{proof}
Cromwell \cite{homogeneous} showed that alternating link diagrams are homogeneous. 
From Theorem~$\ref{theorem:positivity2}$, an alternating link $L$ is positive if and only if $L$ satisfies $s(L)=2g(L)+\sharp L-1$. 
\end{proof}
The following was proved by Nakamura \cite{nakamura2}. 
\begin{cor}[{\cite{nakamura2}}] \label{cor:Nakamura}
Let $L$ be a positive and alternating link.
Then any reduced alternating diagram of $L$ is positive. 
\end{cor}
\begin{proof}
It is known that a reduced alternating link diagram $D$ of $L$ are homogeneous. 
If $L$ is positive, we have $s(L)=2g(L)+\sharp L-1$. 
By Lemma~$\ref{lem:positivity}$, the diagram $D$ has no negative crossing, that is, $D$ is positive. 
\end{proof}
%
%
%
\section{The $s$-invariants of almost positive links}\label{sec:almost-positive-link}
In this section, we compute the $s$-invariants of almost positive links. 
\par
A diagram is {\it almost positive} if it has exactly one negative crossing. 
Then, we can see that an almost positive link is not positive and is represented by an almost positive diagram. 
\par
It is known that, for any link $L$, we obtain $s(L)\leq 2g_{*}(L)+\sharp L-1$. 
On the other hand, for an almost positive link diagram $D$ of a non-split link $L$, we can check $H_{Kh}^{0,j}(L)=0$ if $j<-O(D)+w(D)=2g(D)+\sharp L-4$, where $H_{Kh}^{i,j}(L)$ is the Khovanov homology of L \cite{khovanov1} and $g(D)$ is the genus of the Seifert surface obtained from $D$ by Seifert's algorithm. 
Hence, we obtain 
\begin{align*}
2g(D)+\sharp L-3 \leq s(L) &\leq 2g_{*}(L)+\sharp L-1\\
&\leq 2g(L)+\sharp L-1\\
&\leq 2g(D)+\sharp L-1. 
\end{align*}
Stoimenow proved that the three-genera of almost positive links are computed from their almost positive diagrams as follows. 
\begin{thm}[{\cite[Corollary~$5$ and the proof of Theorems~$5$ and $6$]{stoimenow1}}]\label{stoimenow1}
Let $D$ be an almost positive diagram of a non-split link $L$ with a negative crossing $p$. 
\begin{enumerate}
\item If there is no (positive) crossing joining the same two Seifert circles of $D$ as the circles which are connected by the negative crossing $p$, we have $g(L)=g(D)$ (see the left of Figure~$\ref{fig:negative}$). 
\item If there is a (positive) crossing joining the same two Seifert circles of $D$ as the circles which are connected by the negative crossing $p$, we have $g(L)=g(D)-1$ (see the right of Figure~$\ref{fig:negative}$). 
\end{enumerate}
\end{thm}
\begin{figure}[!h]
\begin{center}
\includegraphics[scale=0.55]{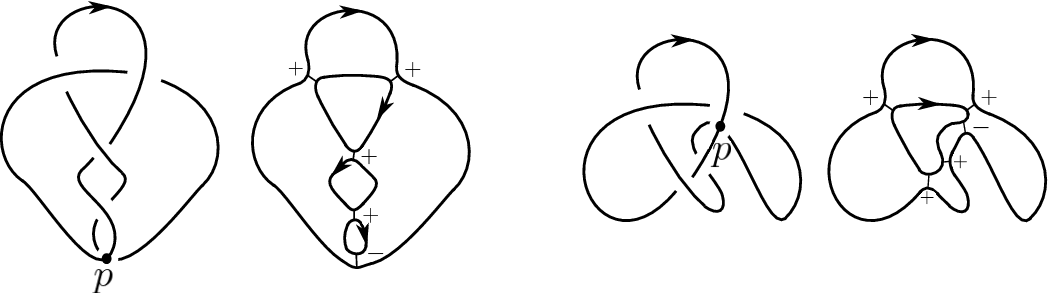}
\end{center}
\caption{In the left picture, there is no crossing joining the same two Seifert circles as the two circles which are connected by the negative crossing $p$. 
In the right picture, there is a crossing joining the same two Seifert circles as the two circles which are connected by the negative crossing $p$.}
\label{fig:negative}
\end{figure}
By the same discussion as \cite{tagami3}, we can compute the $s$-invariants of almost positive links as follows. 
\begin{thm}\label{link-rem}
Let $D$ be an almost positive diagram of a link $L$ with negative crossing $p$. 
\begin{enumerate}
\item If there is no crossing joining the same two Seifert circles of $D$ as the two circles which are connected by the negative crossing $p$, we obtain 
\begin{align*}
s(L) +1-\sharp L=&2g_{*}(L)=2g(L)=2g(D),  
\end{align*}
\item otherwise, we obtain 
\begin{align*}
s(L) +1-\sharp L=&2g_{*}(L)=2g(L)=2g(D)-2.  
\end{align*}
\end{enumerate}
\end{thm}
\begin{proof}
Let $D_{+}$ be the positive diagram obtained from $D$ by the crossing change at $p$ and $L_{+}$ the link represented by $D_{+}$. 
By well known properties of the $s$-invariant, 
we obtain 
\begin{align}
s(L_{+})-2&\leq s(L)\leq s(L_{+}), \label{1}\\
|s(L)|&\leq 2g_{*}(L)+\sharp L-1\leq 2g(L)+\sharp L-1  \label{2}, \\
s(L_{+})+1-\sharp L&=2g_{*}(L_{+})=2g(L_{+})=2g(D_{+})(=2g(D)). \label{3}
\end{align}
\par
$(1)$ Suppose that there is no (positive) crossing joining the same two Seifert circles as the circles which are connected by the negative crossing $p$: 
By $(\ref{1})$, we can see that $s(L)=s(L_{+})$ or $s(L_{+})-2$. 
By Lemma~$\ref{lem1}$ below and $(\ref{3})$, we have $s(L)\neq 2g(D)+\sharp L-3=s(L_{+})-2$. Hence, we obtain $s(L)=s(L_{+})=2g(D)+\sharp L-1$. 
By $(\ref{2})$, we have 
\begin{align*}
2g(D)+\sharp L-1=s(L)\leq 2g_{*}(L)+\sharp L-1\leq 2g(L)+\sharp L-1\leq  2g(D)+\sharp L-1. 
\end{align*}
\par
$(2)$ Suppose that there is a (positive) crossing joining the same two Seifert circles as the circles which are connected by the negative crossing $p$: 
By Theorem~$\ref{stoimenow1}$, $(\ref{2})$ and $(\ref{3})$, we obtain 
\begin{align*}
2g(D)+\sharp L-3&=s(L_{+})-2\\
&\leq s(L)\\
&\leq 2g_{*}(L)+\sharp L-1\\
&\leq 2g(L)+\sharp L-1= 2g(D)+\sharp L-3. 
\end{align*}
\end{proof}
\begin{proof}[Proof of Corollary~$\ref{cor2}$]
By Theorem~$\ref{theorem:positivity2}$, a homogeneous link $L$ satisfying $s(L)=2g_{*}(L)+\sharp L-1=2g(L)+\sharp L-1$ is a positive link. 
By Theorem~$\ref{link-rem}$, all almost positive links satisfy $s(L)=2g_{*}(L)+\sharp L -1=2g(L)+\sharp L -1$. 
Hence any almost positive link is not homogeneous. 
\end{proof}
\begin{proof}[Proof of Theorem~$\ref{main2}$]
Theorem~$\ref{main2}$ follows from Theorems~$\ref{theorem:positivity}$ and $\ref{link-rem}$. 
\end{proof}
\begin{lem}[{\cite[Lemma 3.4]{tagami3}}]\label{lem1}
Let $D$ be an almost positive link diagram of a non-split link $L$ with a negative crossing $p$. 
If there is no (positive) crossing of $D$ joining the same two Seifert circles as the circles which are connected by the negative crossing $p$, 
we have $H_{Kh}^{0, 2g(D)+\sharp L-4}(L)=0$, where $H_{Kh}^{i,j}(L)$ is the Khovanov homology of L and $\sharp L$ is the number of the components of $L$. 
\end{lem}
%

\section{Strong quasipositivities of almost positive knots with up to $12$ crossings}\label{appendix}
In order to present evidence towards an affirmative answer to Stoimenow's question (Question~$\ref{conj:stoimenow}$), 
in this section, we check the strong quasipositivities of almost positive knots with up to $12$ crossings. 
In Subsection~$\ref{sec:table-positive}$, we find all knots which are or may be almost positive with up to $12$ crossings. 
In Subsection~$\ref{sec:table2}$, we check the strong quasipositivities of these knots. 
\subsection{The positivities and almost positivities of knots up to $12$ crossings}\label{sec:table-positive}
In this subsection, we consider the positivities and almost positivities of knots with up to $12$ crossings. 
Here, we call a knot positive if the knot or the mirror image of the knot has a positive diagram. 
By using Proposition~$\ref{prop:almost}$, Theorems~$\ref{theorem:positivity2}$, $\ref{almost2}$--$\ref{thm:almost3}$ and $\ref{homo}$--$\ref{kh-thick}$, and Lemma~$\ref{n638}$ below,   
we can determine the positivities and almost positivities of knots with up to $12$ crossings 
except for $12_{n148}$, $12_{n276}$, $12_{n329}$, $12_{n366}$, $12_{n402}$, $12_{n528}$ and $12_{n660}$, which have almost positive diagrams 
(here we used KnotInfo \cite{knot_info} due to Cha and Livingston, and the Mathematica Package ``KnotTheory"\cite{Bar-Natan-2}). 
See Table~$\ref{table1}$. 
\begin{thm}[{\cite[Corollary~$1.7$]{almostpositive}}, {\cite[Corollary~$6.1$]{almost_stoimenow}}]\label{almost2}
Nontrivial almost positive links have negative signature.
\end{thm}
\begin{thm}[{\cite[Corollaries~$2.1$ and $2.2$]{homogeneous}}, {\cite{positive-conway}}]\label{almost3}
If $L$ is an almost positive link or a positive link, then all coefficients of its Conway polynomial are non-negative. 
\end{thm}
\begin{thm}[{\cite[Theorem~$6$]{stoimenow1}}]\label{thm:almost3}
If $L$ is an almost positive link, then 
\begin{align*}
\operatorname{maxdeg}_{z} \nabla_{L}(z)=\operatorname{maxdeg}_{z}P_{L}(v, z)=1-\chi(L),  
\end{align*}
where $\nabla_{L}$ is the Conway polynomial and $P_{L}(v, z)$ is the HOMFLYPT polynomial. 
\end{thm}
\begin{prop}[{\cite[Proposition~$6.2$]{stoimenow_almostpositive3}}]\label{prop:almost}
Let $K$ be an almost positive knot with $g(K)\geq 3$. Then its signature $\sigma (K)$ is smaller than or equal to $-4$. 
\end{prop}
\begin{thm}[{\cite[Corollary $5.1$]{homogeneous}}]\label{homo}
If $L$ is a homogeneous link and the coefficient of the maximal degree term of its Conway polynomial is $\pm 1$, then the number of the crossings of a homogeneous diagram of $L$ is at most $2\cdot maxdeg_{z} \nabla_{L}(z)$, where $maxdeg_{z} \nabla_{L}(z)$ is the maximal degree of the Conway polynomial of $L$. 
In particular, the minimal crossing number of $L$ is at most $2\cdot maxdeg_{z} \nabla_{L}(z)$. 
\end{thm}
\begin{thm}[{\cite[Theorem~$1.4$]{positive-genustwo}}]\label{kh-thick}
Positive knots up to genus two are quasialternating. 
\end{thm}
For the definition of quasialternating links, see \cite{quasi-alternating}. 
\begin{lem}\label{n638}
The knot $12_{n638}$ is a positive knot. 
\end{lem}
\begin{proof}
See Figure~$\ref{12_n638}$. 
\end{proof}
\begin{figure}[!h]
\begin{center}
\includegraphics[scale=0.75]{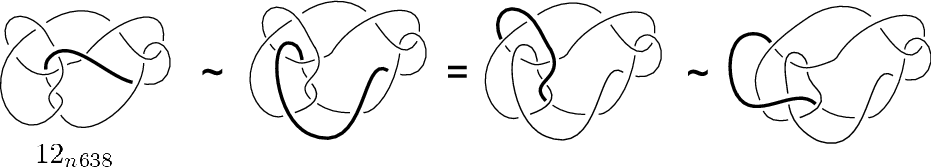}
\end{center}
\caption{The knot $12_{n638}$ has a positive diagram. }
\label{12_n638}
\end{figure}
%
%
\begin{rem}
In the above process, we find some almost positive knots, $10_{145}$, $12_{n149}$, $12_{n332}$, $12_{n404}$, $12_{n432}$ and $12_{n642}$. 
They have almost positive diagrams, and 
$10_{145}$, $12_{n404}$ and $12_{n642}$ are not homogeneous by Theorem~$\ref{homo}$. 
The knots $12_{n149}$, $12_{n332}$ and $12_{n432}$ are not positive by Theorem~$\ref{kh-thick}$. 
\end{rem}
The knots $12_{n148}$, $12_{n276}$, $12_{n329}$, $12_{n366}$, 
$12_{n402}$, $12_{n528}$ and $12_{n660}$ are either positive or almost positive since they have almost positive diagrams. 
In general, it is hard to check whether given almost positive link diagram represents a positive link or not. 
\begin{question}\label{question:almost}
Are the knots $12_{n148}$, $12_{n276}$, $12_{n329}$, $12_{n366}$, 
$12_{n402}$, $12_{n528}$ and $12_{n660}$ non-positive? (If so, they are almost positive knots.) 
\end{question}
\begin{rem}
In \cite[Example~$6.1$]{almost_stoimenow} and \cite[Corollary~$10$]{stoimenow1}, Stoimenow introduced infinitely many almost positive knots. 
\end{rem}
\begin{table}[h]
\begin{tabular}{|c|c|c|}
\hline
&$\leq 11$ crossings&$12$ crossings \\ \hline
total&$801$&$2176$ \\ \hline
non-positive (negative) knots&$693$&$2031 \leq , \leq2038$ \\ \hline
positive (negative) knots&$108$&$138\leq , \leq 145 $ \\ \hline\hline
almost positive (negative) knots&$1$&$5\leq , \leq  12$ \\ \hline
\end{tabular}
\caption{The positivities of knots with up to $12$ crossings. To determine the almost positivities of some knots, we use Theorem~$\ref{thm:almost3}$ and 
Proposition~$\ref{prop:almost}$. 
The only almost positive knot with up to $11$ crossings is $10_{145}$. 
The knots, $12_{n149}$, $12_{n332}$, $12_{n404}$, $12_{n432}$ and $12_{n642}$ are almost positive. 
Are $12_{n148}$, $12_{n276}$, $12_{n329}$, $12_{n366}$, $12_{n402}$, $12_{n528}$, and $12_{n660}$ almost positive? } 
\label{table1}
\end{table}
%
%
\subsection{Strong quasipositivities of almost positive knots with up to 12 crossings}\label{sec:table2}
We check the strong quasipositivities of almost positive knots with up to 12 crossings. 
In this section, we call a knot strongly quasipositive if the knot or the mirror image of the knot is strongly quasipositive. 
\par
From Table~$\ref{table1}$, the $6$ knots, $10_{145}$, $12_{n149}$, $12_{n332}$, $12_{n404}$, $12_{n432}$ and $12_{n642}$ are almost positive. 
In addition, the $7$ knots, $12_{n148}$, $12_{n276}$, $12_{n329}$, $12_{n366}$, $12_{n402}$, $12_{n528}$, and $12_{n660}$ may be almost positive, and other knots with up to $12$ crossings are not almost positive. 
From Lemmas~$\ref{lem:almost-strong2}$ and $\ref{lem:almost-strong1}$ below and Table~$\ref{table1}$, we obtain the following proposition. 
The proposition is evidence towards an affirmative answer to Question~$\ref{conj:stoimenow}$. 
\begin{prop}
All almost positive knots with up to $12$ crossings are strongly quasipositive. 
\end{prop}
\begin{lem}\label{lem:almost-strong2}
The $9$ knots, $10_{145}$, $12_{n148}$, $12_{n276}$, $12_{n329}$, $12_{n366}$, $12_{n402}$, $12_{n528}$, $12_{n642}$ and $12_{n660}$ are strongly quasipositive. 
\end{lem}
\begin{proof}
It is known that these knots are fibered (KnotInfo \cite{knot_info}). 
These knots are positive or almost positive because they have almost positive diagrams. 
Note that positive links are strongly quasipositive (see \cite{Nakamura} and \cite{Rudolph}). 
By Theorem~$\ref{thm:fiber-almost}$ below, these knots are strongly quasipositive. 
\end{proof}
\begin{thm}\label{thm:fiber-almost}
All fibered almost positive knots are strongly quasipositive. 
\end{thm}
\begin{proof}
Let $K$ be a fibered almost positive knot and $D$ be an almost positive diagram. 
Obviously, the diagram $D$ has a $*$-product decomposition whose factors are some positive diagrams $D_{1},\dots, D_{n-1}$ and one special almost positive diagram $D_{n}$. 
Let $S$ and $S_{i}$ be the Seifert surfaces obtained from $D$ and $D_{i}$, respectively. 
Note that $S_1, \dots, S_{n-1}$ are quasipositive surfaces (see \cite{Nakamura} and \cite{Rudolph}). 
We consider two cases as follows. 
\par
(i) Suppose that there is no crossing joining the same two Seifert circles of $D$ as the two circles which are connected by the negative crossing: 
In this case, by Theorem~$\ref{stoimenow1}$, the surface $S$ has minimal genus. 
In particular, the surface is the fiber surface. 
By Gabai's results \cite{gabai2, gabai1}, the Seifert surface $S_{i}$ is also the fiber surface. 
Then, by Goda-Hirasawa-Yamamoto's result \cite[Corollary~$1.8$]{fiber-almostalt}, the fiber surface $S_{n}$ is a plumbing of positive Hopf bands. 
Since the positive Hopf band is a quasipositive surface and plumbings preserve the quasipositivites of surfaces \cite{Rudolph3}, the surface $S_{n}$ is quasipositive. 
Hence, the surface $S$ is quasipositive since it is a Murasugi sum of the quasipositive surfaces $S_{1}, \dots, S_{n}$ (see \cite{Rudolph3}). 
In particular, the knot $K$ is strongly quasipositive. 
\par
(ii) In other cases, by the same discussion as Theorem~$\ref{link-rem}$ (2), we have 
\begin{align*}
\tau(K)=g_{*}(K)=g(K)=g(D)-1, 
\end{align*}
where $\tau(K)$ is Ozsv{\' a}th-Szab{\' o}'s $\tau$-invariant of $K$. 
Hedden \cite[Theorem~1.2]{Hedden2} proved that for a fibered knot $K'$, 
the knot is strongly quasipositive if and only if $\tau(K')=g_{*}(K')=g(K')$. 
Hence, $K$ is strongly quasipositive. 
\end{proof}
\begin{lem}\label{lem:almost-strong1}
The knots $12_{n149}$, $12_{n332}$, $12_{n404}$ and $12_{n432}$ (see Figure~$\ref{four-knots}$) are strongly quasipositive. 
\end{lem}
\begin{proof}
Firstly, we check the strong quasipositivity of $12_{n149}$. 
As the pictures in Figure~$\ref{12n149}$ show, the canonical Seifert surface of a positive knot diagram is obtained from a Seifert surface of $12_{n149}$ by two deplumbings. 
Note that the canonical Seifert surface of a positive knot diagram is quasipositive (see \cite{Nakamura} and \cite{Rudolph}). 
Since plumbings and deplumbings preserve the quasipositivities of surfaces (see \cite{Rudolph3}), this Seifert surface of $12_{n149}$ is quasipositive. 
Hence $12_{n149}$ is strongly quasipositive. 
By the same discussion, we can prove that $12_{n332}$, $12_{n404}$ and $12_{n432}$ are strongly quasipositive (see Figures~$\ref{12n332}$, $\ref{12n404}$ and $\ref{12n432}$). 
\end{proof}
\begin{figure}[h]
\begin{center}
\includegraphics[scale=0.6]{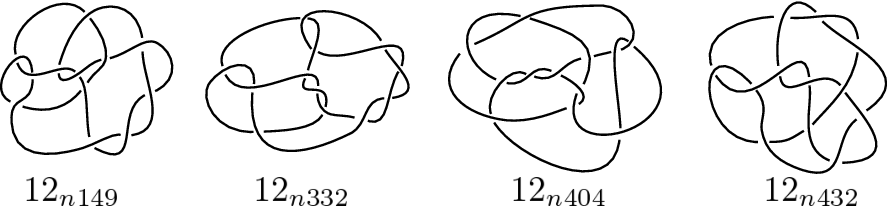}
\caption{$12_{n149}$, $12_{n332}$, $12_{n404}$ and $12_{n432}$. }\label{four-knots}
\end{center}
\end{figure}
\begin{figure}[h]
\begin{center}
\includegraphics[scale=0.795]{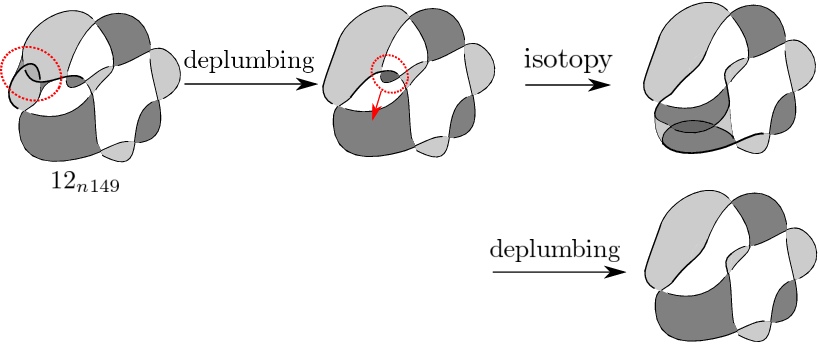}
\caption{The top left picture is the canonical Seifert surface of an almost positive diagram of $12_{n149}$. These pictures show that the Seifert surface is quasipositive. }\label{12n149}
\end{center}
\end{figure}
\begin{figure}[h]
\begin{center}
\includegraphics[scale=0.8]{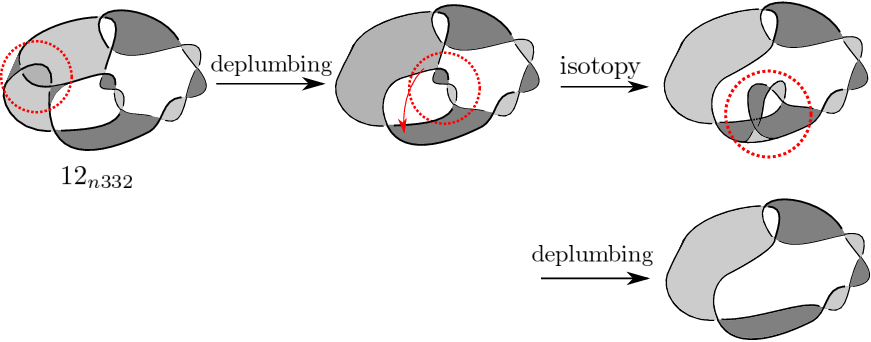}
\caption{A proof of the strong quasipositivity of $12_{n332}$. }\label{12n332}
\end{center}
\end{figure}
\begin{figure}[h]
\begin{center}
\includegraphics[scale=0.6]{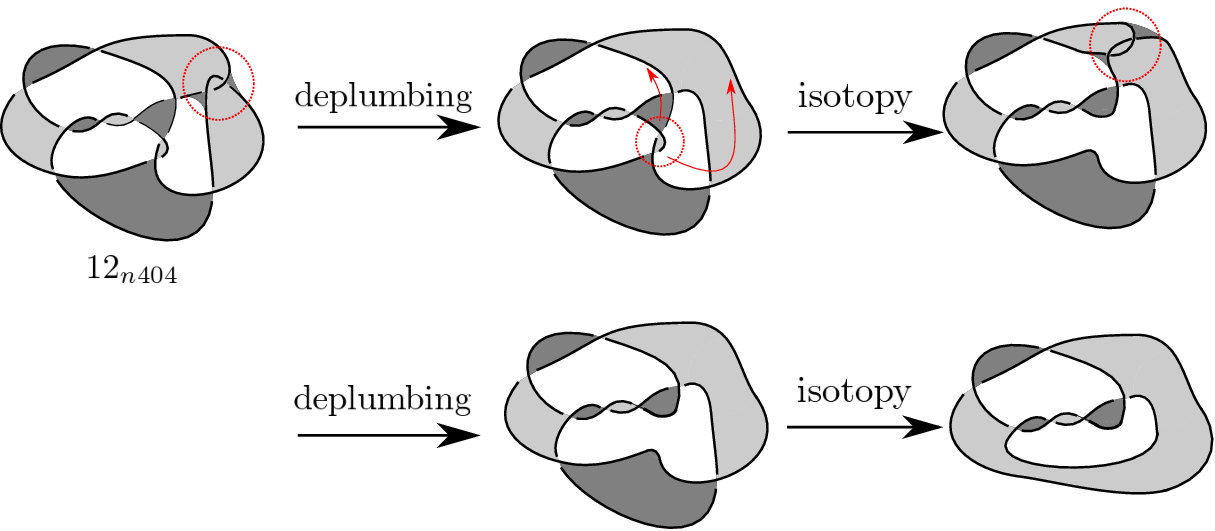}
\caption{A proof of the strong quasipositivity of $12_{n404}$. }\label{12n404}
\end{center}
\end{figure}
\begin{figure}[h]
\begin{center}
\includegraphics[scale=0.62]{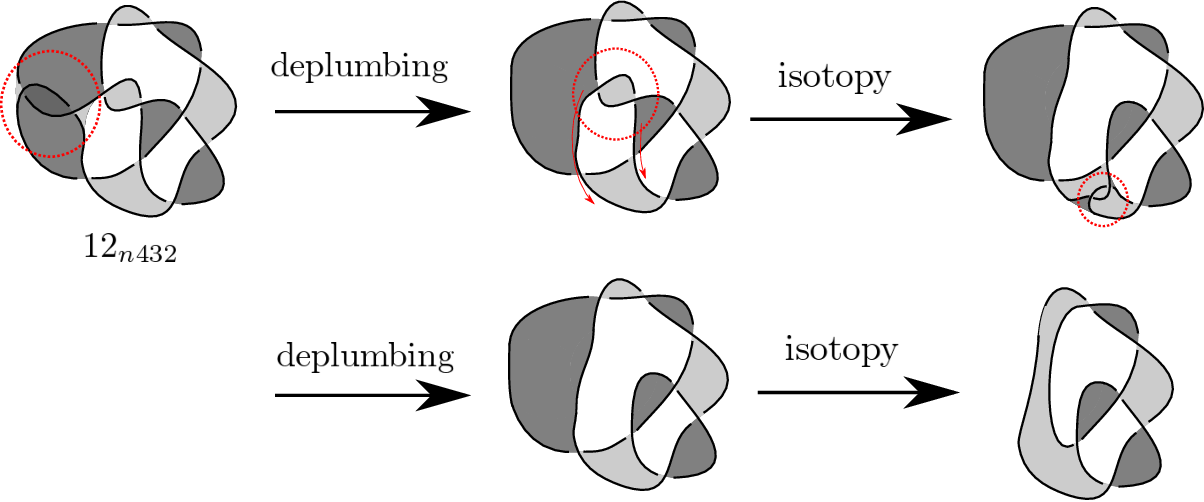}
\caption{A proof of the strong quasipositivity of $12_{n432}$. }\label{12n432}
\end{center}
\end{figure}
%
%
%
\section{Infinitely many counterexamples of Kauffman's conjecture on pseudo-alternating links and alternative links. }\label{sec:app}
In this section, we give infinitely many counterexamples of Kauffman's conjecture on pseudo-alternating links and alternative links. 
\par
At first, we recall the definition of pseudo-alternating links \cite{pseudoalternating}. 
A {\it primitive flat surface} is the canonical Seifert surface obtained from a special alternating diagram by Seifert's algorithm. 
A {\it generalized flat surface} is an orientable surface obtained from some primitive flat surfaces by Murasugi sum along their Seifert disks (for example, see the bottom figure in Figure~$\ref{Kn}$). 
Then, a link is {\it pseudo-alternating} if it bounds a generalized flat surface. 
\par
Next, we recall the definition of alternative links \cite{formal_knot}. 
For a link diagram $D$, the {\it spaces of $D$} are the connected components of the complement of the Seifert circles of $D$ in $S^{2}$. 
We draw an edge joining two Seifert circles at the place where a crossing of $D$ connects the circles. 
Moreover, we assign the sign ``$+$" (resp.~``$-$") to an edge if the crossing corresponding to the edge is positive (resp.~negative). 
Then, a diagram $D$ is {\it alternative} if for each space $X$ of $D$, all the edges in $X$ have the same sign. 
\par
From the definitions, we have the following. 
\begin{cor}
All alternative links are homogeneous. 
All homogeneous links are pseudo-alternating. 
\end{cor}
Kauffman conjectured that all pseudo-alternating links are alternative. 
\begin{conjecture}[\cite{formal_knot}]\label{kauffman_conjecture}
All pseudo-alternating links are alternative. 
\end{conjecture}
However, this conjecture is false. 
In fact Silvero \cite{silvero} introduced two counterexamples, $10_{145}$ and $L9n18$. 
\par
Here, we prove that the infinitely many almost positive knots introduced by Stoimenow (which contains $10_{145}$) are counterexamples for this conjecture. 
\begin{prop}\label{counterexample}
Let $K_{n}$ be the knot depicted in Figure~\ref{Kn2}. Then, $K_{n}$ is non-alternative and is pseudo-alternating. 
\end{prop}
\begin{proof}
Stoimenow \cite[Example ~$6.1$]{almost_stoimenow} proved that $K_{n}$ is almost positive. 
By Corollary~$\ref{cor2}$, the knot $K_{n}$ is not homogeneous, in particular, not alternative. 
On the other hand, by Figure~$\ref{Kn}$, the knot $K_{n}$ bounds a generalized flat surface. 
\end{proof}
\begin{proof}[Proof of Proposition~$\ref{cor3}$]
This follows from Proposition~\ref{counterexample}. 
\end{proof}
Finally, we give two questions. 
\begin{question}
Are all almost positive links pseudo-alternating?
\end{question}
\begin{question}
Are all homogeneous links alternative?
\end{question}
\begin{figure}[h]
\begin{center}
\includegraphics[scale=0.9]{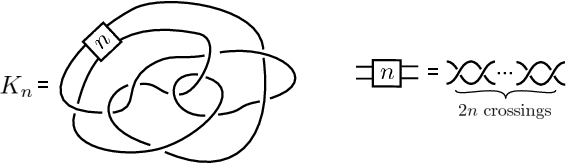}
\caption{The knot $K_{n}$ introduced by Stoimenow \cite[Example ~$6.1$]{almost_stoimenow}, where $n\geq 0$ is the number of the full twists. Stoimenow proved that $K_{n}$ is almost positive. }\label{Kn2}
\end{center}
\end{figure}
\begin{figure}[h]
\begin{center}
\includegraphics[scale=0.9]{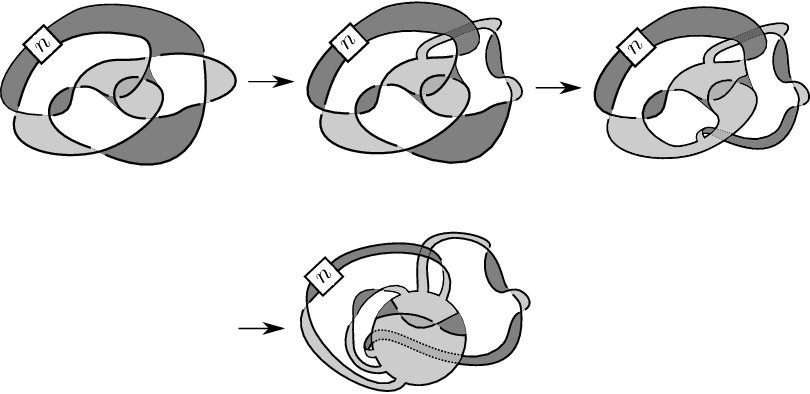}
\caption{The top left picture is a Seifert surface of $K_{n}$. By isotopy, the surface changes into the bottom surface which is a generalized flat surface. }\label{Kn}
\end{center}
\end{figure}
\noindent{\bf Acknowledgements: }
The first author was supported by JSPS KAKENHI Grant numbers 13J05998 and 16K17597. 
Some parts of this paper were written in 2011 during his stay at Indiana University. 
He deeply thanks Paul Kirk and Chuck Livingston for their hospitality. 
He also thanks Kokoro Tanaka for telling him Corollary~\ref{cor:Nakamura} a few years ago.
The second author was supported by JSPS KAKENHI Grant numbers 13J01362, 15J01087 and 16H07230. 
He would like to thank K{\'a}lm{\'a}n Tam{\'a}s for his helpful comments and warm encouragement. 
The authors would like to thank the referee for his/her kind and helpful comments. 
%
\par
\ 
\par
\noindent
{\bf Comments: }
We note that this is the published version \cite{abe-tagami1-1} of \cite{abe-tagami1}. 
In the previous version \cite{abe-tagami1}, we also gave a survey for the Khovanov homology, the Lee homology, Rasmussen and Beliakova-Wehrli's invariant and Kawamura-Lobb's inequality. 
%
\bibliographystyle{amsplain}
\bibliography{mrabbrev,tagami}
\end{document}